\documentclass[reqno, 12pt]{amsart}
\makeatletter
\let\origsection=\section 
\def\section{\@ifstar{\origsection*}{\mysection}} 
\def\mysection{\@startsection{section}{1}\z@{.7\linespacing\@plus\linespacing}{.5\linespacing}{\normalfont\scshape\centering\S\hspace{1pt}}}
\makeatother        

\usepackage{amsmath,amssymb,amsthm}
\usepackage{mathrsfs} 
\usepackage{mathabx}\changenotsign
\usepackage{dsfont}

\usepackage{xcolor}
\usepackage[backref]{hyperref} 
\hypersetup{
	colorlinks,
	linkcolor={red!60!black},
	citecolor={green!60!black},
	urlcolor={blue!60!black}
}
\usepackage[open,openlevel=2,atend]{bookmark}

\usepackage[abbrev,msc-links,backrefs]{amsrefs}

\usepackage{doi}

\renewcommand{\PrintDOI}[1]{\doi{#1}}

\usepackage[T1]{fontenc}
\usepackage{lmodern}
\usepackage[babel]{microtype}
\usepackage[english]{babel}

\usepackage{geometry}
\numberwithin{equation}{section}
\numberwithin{figure}{section}

\usepackage{enumitem}

\def\RMlabel{\upshape(\Roman*)}

\def\Alabel{\upshape({\itshape \Alph*\,})}
\def\nlabel{\upshape({\itshape \arabic*\,})}
\renewcommand\labelenumi{(\roman{enumi})}
\renewcommand\theenumi\labelenumi

\theoremstyle{plain}
\newtheorem{thm}{Theorem}[section]

\newtheorem{prop}[thm]{Proposition}

\newtheorem{cor}[thm]{Corollary}
\newtheorem{lemma}[thm]{Lemma}

\theoremstyle{definition}

\newtheorem{definition}[thm]{Definition}

\newtheorem{conj}[thm]{Conjecture}

\let\eps=\varepsilon
\let\theta=\vartheta
\let\rho=\varrho
\let\phi=\varphi

\def\PP{\mathds P}
\def\EE{\mathds E}

\def\ind{\mathds 1}

\def\cB{{\mathcal B}}

\def\cH{{\mathcal H}}
\def\cK{{\mathcal K}}
\def\cL{{\mathcal L}}

\def\cR{{\mathcal R}}
\def\cS{{\mathcal S}}
\def\cT{{\mathcal T}}

\def\ccP{{\mathscr{P}}}

\DeclareMathOperator{\ex}{ex}

\DeclareMathOperator{\Ind}{ind}
\DeclareMathOperator{\Ext}{Ext}

\let\polishlcross=\l
\def\l{\ifmmode\ell\else\polishlcross\fi}

\makeatletter
\def\moverlay{\mathpalette\mov@rlay}
\def\mov@rlay#1#2{\leavevmode\vtop{   \baselineskip\z@skip \lineskiplimit-\maxdimen
		\ialign{\hfil$\m@th#1##$\hfil\cr#2\crcr}}}
\newcommand{\charfusion}[3][\mathord]{
	#1{\ifx#1\mathop\vphantom{#2}\fi\mathpalette\mov@rlay{#2\cr#3}}
	\ifx#1\mathop\expandafter\displaylimits\fi
}
\makeatother

\newcommand{\vrhup}[1]{\scaleobj{0.6}{\scalerel*{\rightharpoonup}{#1}}}
\newcommand{\nrhup}{\mathord{\scaleobj{0.6}{\scalerel*{\rightharpoonup}{x}}}}
\newcommand{\wrhup}{\scaleobj{0.6}{\scalerel*{\rightharpoonup}{W}}}
\def\vseq#1{\ThisStyle{  \mathord{\vbox{\offinterlineskip\ialign{    \hfil##\hfil\cr
					$\SavedStyle{}_{\smash{\vrhup#1}}$\cr
					\noalign{\kern-0.7\scriptspace}
					$\SavedStyle#1$\cr}}}}}
\def\seq#1{\ThisStyle{  \mathord{\vbox{\offinterlineskip\ialign{    \hfil##\hfil\cr
					$\SavedStyle{}_{\smash{\nrhup}}$\cr
					\noalign{\kern-0.5\scriptspace}
					$\SavedStyle#1$\cr}}}}}
\def\wseq#1{\ThisStyle{  \mathord{\vbox{\offinterlineskip\ialign{    \hfil##\hfil\cr
					$\SavedStyle{}_{\smash{\wrhup#1}}$\cr
					\noalign{\kern-0.7\scriptspace}
					$\SavedStyle#1$\cr}}}}}

\let\setminus=\smallsetminus
\let\emptyset=\varnothing

\let\epsilon=\varepsilon

\let\to=\lra

\makeatletter
\newcommand{\oset}[3][0ex]{\mathrel{\mathop{#3}\limits^{
			\vbox to#1{\kern-2.1\ex@
				\hbox{$\scriptstyle#2$}\vss}}}}
\makeatother

\DeclareMathSymbol{*}{\mathbin}{symbols}{"03}

\makeatletter
\newcommand{\pushright}[1]{\ifmeasuring@#1\else\omit\hfill$\displaystyle#1$\fi\ignorespaces}
\newcommand{\pushleft}[1]{\ifmeasuring@#1\else\omit$\displaystyle#1$\hfill\fi\ignorespaces}
\makeatother

\usepackage{tikz}
\usepackage{scalerel}
\newsavebox\eebox
\savebox\eebox{\tikz{
		\draw[black,fill=black] (90:1) circle (.35);
		\draw[black,fill=black] (210:1) circle (.35);
		\draw[black,fill=black] (330:1) circle (.35);
		\draw[black,line width=0.28cm ] (90:1) -- (330:1);
		\draw[black,line width=0.28cm ] (90:1) -- (210:1);
		\draw[opacity=0] (0:1.2) circle (0.1);
}}

\newsavebox\epbox
\savebox\epbox{\tikz{
		\draw[black,fill=black] (0.866,1) circle (.35);
		\draw[black,fill=black] (-0.866,1) circle (.35);
		\draw[black,fill=black] (210:1) circle (.35);
		\draw[black,fill=black] (330:1) circle (.35);
		\draw[black,line width=0.28cm ] (0.866,1) -- (330:1);
		\draw[black,line width=0.28cm ] (-0.866,1) -- (210:1);
		\draw[opacity=0] (0:1.2) circle (0.1);
}}

\let\Jn=J

\newcommand{\ovl}{\overline}

\newcommand\strongarrow{\mathrel{\overset{\makebox[0pt]{\mbox{\normalfont\tiny\sffamily ind}}}{\rightarrow}}}

\begin{document}
	\
	\title{Coloring triangles in graphs}
   \author{Ayush Basu}
\author{Vojt\v{e}ch R\"odl}
\address{Department of Mathematics, Emory University, 
    Atlanta, GA, USA}
\email{\{abasu|vrodl\}@emory.edu}
    \author{Marcelo Sales}
    \address{Department of Mathematics, University of California, Irvine, CA, USA}
\email{mtsales@uci.edu}

\thanks{The first and second authors were partially supported by NSF grants DMS 1764385 and 2300347. The third author was supported by US Air Force grant FA9550-23-1-0298.}

    \begin{abstract}
    We study quantitative aspects of the following fact: For every graph $F$, there exists a graph $G$ with the property that any $2$-coloring of the triangles of $G$ yields an induced copy of $F$, in which all triangles are monochromatic. We define the Ramsey number $R_{\Ind}^{\Delta}(F)$ as the smallest size of such a graph $G$. Although this fact has several proofs, all of them provide tower-type bounds. We study the number $R_{\Ind}^{\Delta}(F)$ for some particular classes of graphs $F$.
    \end{abstract}
 \maketitle

\section{Introduction}

For a natural number $N$, we set $[N] = \{1, \ldots, N\}$. Given a set $X$ and a positive integer $k$, we write $X^{(k)} = \{e \subseteq X :\: |e| = k\}$ for the set of all $k$-element subsets of $X$. A $k$-uniform hypergraph (or $k$-graph) is a pair $H = (V, E)$, where $V$ is the set of vertices and $E \subseteq V^{(k)}$ is the set of edges.
\\[0.3cm]
Given two $k$-graphs $F$ and $H$, we write $H \rightarrow F$ if every $2$-coloring of the edges of $H$ yields a monochromatic copy of $F$. The \emph{Ramsey number $R(F)$ of the $k$-graph $F$} is the smallest integer $N$ such that there exists a graph $H$ on $N$ vertices satisfying $H \rightarrow F$. The existence of $R(F)$ (Ramsey's theorem) was established by Ramsey \cite{R29} and later rediscovered by Erd\H{o}s and Szekeres \cite{ES35}.
\\[0.3cm]
Another possible direction is to study extensions of Ramsey's theorem when we impose that the monochromatic copy has additional structural properties. One of these extensions is the induced Ramsey theorem. We say that a $k$-graph $F$ is an \emph{induced} copy of $H$ if there exists an injective embedding $\phi: V(F) \rightarrow V(H)$ such that $\{\phi(x),\phi(y)\} \in H$ if and only if $\{x,y\} \in F$, i.e., edges are mapped to edges and non-edges to non-edges. For a $k$-graph $F$, the \emph{induced Ramsey number $R_{\Ind}(F)$} is the smallest integer $N$ for which there exists a $k$-graph $H$ such that every $2$-coloring of the edges of $H$ yields an induced monochromatic copy of $F$ (denoted by the relation $H \strongarrow F$).
\\[0.3cm]
The existence of $R_{\Ind}(F)$ was shown for graphs and hypergraphs independently by a series of authors \cites{D75, EHP75, rodlthesis73, AH78, NR77}. These original proofs provided no quantitative bounds on $R_{\Ind}(F)$. However, Erd\H{o}s \cite{E84} conjectured that for graphs, the number $R_{\Ind}(F)$ should be exponential in the number of vertices of $F$. The best current upper bound is due to Conlon, Fox, and Sudakov \cite{CFS12}, who showed that for every graph $F$ on $n$ vertices,
\begin{align}\label{eq:inducedgraphs}
    R_{\Ind}(F) \leq 2^{cn\log n},
\end{align}
improving a result from \cite{KPR98}. The problem for hypergraphs was studied in \cite{CDFRS17}, where in particular the authors proved that for a $3$-graph $F$ on $n$ vertices, the bound $R_{\Ind}(F)\leq 2^{2^{cn}}$ holds.
\\[0.3cm]
We are interested in a variant of the induced Ramsey theorem proved in \cites{Deuber75, NR75}. Given two graphs $G$ and $F$, we say that
\begin{align*}
    G \strongarrow (F)^{\Delta}
\end{align*}
if, for any $2$-coloring of the triangles in $G$, there exists an induced copy of $F$ such that all its triangles are monochromatic. Let $R_{\Ind}^{\Delta}(F)$ be the smallest integer $N$ such that there exists a graph $G$ on $N$ vertices with $G \strongarrow (F)^{\Delta}$. The existence of $R_{\Ind}^{\Delta}(F)$ was proven in \cites{NR75, Deuber75, NR82, BR16}, though none of those proofs provided a numerical bound.\\[0.3cm]
For a graph $F$, let $\cK_3(F)$ be the $3$-graph on the vertex set $V(F)$ whose edges are the triplets of vertices that form triangles in $F$. We remark that in order to show that $G \strongarrow (F)^{\Delta}$, it is not sufficient to find a $3$-graph $H^{(3)}$ with $H^{(3)} \strongarrow \cK_3(F)$. Indeed, we require the monochromatic copy $\cK$ of $\cK_3(F)$ to be not only induced in $H^{(3)}$, but also that having the property that for every pair of vertices of $V(\cK)$ the only edges containing them in $H^{(3)}$ are those in $\cK$. A copy $\cK$  with such property is called \emph{a strongly induced copy} of $\cK_3(F)$. Therefore, in order for $G$ to satisfy $G\strongarrow (F)^{\Delta}$, the hypergraph $H=\cK_3(G)$ needs to have a strongly induced monochromatic copy of $\cK_3(F)$. This additional requirement implies that $R_{\Ind}(F) \leq R_{\Ind}^{\Delta}(F)$.
\\[0.3cm]
The main contribution of the paper is to provide bounds for $R_{\Ind}^{\Delta}(F)$ for certain families of graphs. The first result is an upper bound when $F$ is a graph whose triangles form a linear $3$-graph. 


\begin{thm}\label{thm: main}
If $F$ is a graph on $n$ vertices such that $\cK_3(F)$ is linear, then
\begin{align*}
    R_{\Ind}^{\Delta}(F)\leq 2^{2^{2^{40n}}}.
\end{align*}
\end{thm}

We remark that since $\cK_3(F)$ is the $3$-graph of triangles of $F$, then $\cK_3(F)$ is linear if and only if $\cK_3(F)$ has girth at least $4$. The second result concerns graphs where $\cK_3(F)$ is a family of bipartite $3$-graphs. Let $\cB_n \subseteq 2^{[n]^{(2)}}$ be the family of graphs $F$ on $n$ vertices with the property that there exists a bipartition $V(F)=A\cup B$ of the vertices such that $F[A]$ is a triangle free graph, $B$ is an independent set and the bipartite graph between $A$ and $B$ is arbitrary. In this case, we will show that the growth of $R_{\Ind}^{\Delta}(F)$ is at most doubly exponential.

\begin{thm}\label{thm:bipartite}
Let $F \in \cB_n$ be a graph on $n$ vertices. Then
\begin{align*}
    R_{\Ind}^{\Delta}(F)\leq 2^{2^{cn\log n}}
\end{align*}
for a positive constant $c$.
\end{thm}


Finally, our last result shows a polynomial bound if $\cK_3(F)$ is a tight tree. A $3$-graph $T$ is called a \emph{tight tree} if its edges can be ordered as $e_1,\ldots,e_t$ such that for each $i\geq 2$, $e_i$ has a vertex $v_i$ that does not belong to any previous edge and $e_i\setminus\{v_i\}$ is contained in $e_j$ for some $j<i$.

\begin{prop}\label{prop:tree}
    Let $F$ be a graph such that $\cK_3(F)$ is a tight tree. Then
    \begin{align*}
        R_{\Ind}^{\Delta}(F)\leq n^4
    \end{align*}
    holds.
\end{prop}

The paper is organized as follows. In Section \ref{sec: properties}, we provide a brief sketch of the proof of Theorem \ref{thm: main} and introduce some notation. Section \ref{sec:embedding} is devoted to prove a version of Ramsey theorem for $2$-colorings of triangles in the random graph, while Section \ref{sec:main} contains the proof of Theorem \ref{thm: main}. In Section \ref{sec:bipartite} and \ref{sec:tree}, we prove Theorem \ref{thm:bipartite} and Proposition \ref{prop:tree}. 


\section{Proof Outline and Notation}
\label{sec: properties}

\subsection{Brief Sketch of Proof of Theorem \ref{thm: main}} Given a graph $F$ such that $\cK_3(F)$ is linear, let,
\begin{align*}
   |V(F)| = n, \quad \text{ and } \quad N = 2^{2^{2^{40n}}}. 
\end{align*}
Our plan is to show that there exists an instance $G$ of $G(N,1/2)$ such that $G\strongarrow (F)^{\Delta}_2$. Here we divide it into the following three steps. 
\begin{enumerate}[label = \RMlabel]
    \item\label{I} In the first step (see Lemma \ref{lem: RamseyGraph}), we show that there is an instance $G$ of $G(N,1/2)$ satisfying the property that for every 2-coloring of $\cK_3(G)$, yields ``many'' monochromatic (say blue) copies of $K_n$. Further, $G$ will also satisfy that every pair of ``large enough'' subsets are $\eps$-regular for a suitable choice of $\eps= \eps(n)$. 
    \item\label{II} In the second step, we consider the $n$-uniform hypergraph $H_n$ obtained by taking the $n$ sets that are the vertices of blue cliques in $G$ as edges, and ``weakly regularise'' (Definition \ref{def: dense}) it (using Lemma \ref{lem: pengrucinski}). We obtain pairwise disjoint vertex sets $U_1,\dots,U_n$ in $V(G)$, such that the crossing edges of $H_n$ with respect to $U_1,\dots, U_n$ are ``weakly regular''. Using this property, we then show that the 3-uniform $n$-partite hypergraph $BT^{(3)}(H_n)$ formed by the blue triangles in the crossing cliques of $H_n$ is also ``weakly regular''.
    \item\label{III} Finally, we use an embedding lemma for embedding linear hypergraphs in weakly regular partitions (see Lemma \ref{lem: embedding}), to find an induced copy of $F$ in $G$ such that all the triangles of $F$ are in $BT^{(3)}(H_n)$. 
\end{enumerate}
In the next subsection, we state briefly state the notation used in the rest of the paper for completeness. 
\subsection{Notation}
Given a graph $G$, a subset $U\subseteq V(G)$, let $e_G(U)$ denote the number of edges in the induced subgraph $G[U]$ and let $d_G(U) = e_G(U)/\binom{|U|}{2}$. Given a $r$-uniform hypergraph $H$ and pairwise disjoint subsets of vertices $V_1,\dots, V_\ell$, let $H[V_1,\dots, V_\ell]$ denote the $\ell$-partite $r$-uniform hypergraph where an edge $e$ of $H$ is in $H[V_1,\dots, V_r]$ if and only if $|e\cap V_\ell|\leq 1$ for every $i\in [\ell]$. Further, for $r=\ell$, let $e_H(V_1,\dots, V_r)$ denote the number of edges in $H[V_1,\dots, V_r]$ and 
$$ d_H(V_1,\dots, V_r) = \frac{e_H(V_1,\dots, V_r)}{|V_1|\cdots |V_r|}.$$
If the underlying graph being considered is clear, we will drop the subscripts in $e_H$, $d_H$, etc.
\section{Embedding Many Monochromatic Cliques in Random Graphs}\label{sec:embedding}
In this section we proceed with the first step (as described in \ref{I}), in the proof of Theorem \ref{thm: main}. We prove Lemma \ref{lem: RamseyGraph}. We fix the following parameters. Given an integer $n$ let
\begin{align*}
        N(n) = N = 2^{2^{2^{40n}}},\quad \text{and} \quad m(n) = m = 2^{2^{4n}}.
\end{align*}
To show that an instance of the random graph $G(N,1/2)$ contains ``many cliques'', we will first show that the random graph $G(m,1/2)\strongarrow (K_n)^{\Delta}_2$ with probability at least 3/4 (see Lemma \ref{lem: PropPRamsey} and Corollary \ref{cor: GnpRamsey}), and thus on average, many subsets of vertices of size $m$ in $G(N,1/2)$ will contain a monochromatic copy of $K_n$. Before we proceed, we introduce a property of the random graph that we will use. 
\subsection{Properties of Random Graphs} 
The following property of the random graph says that all sufficiently large subsets have density roughly $1/2$. Given sets of vertices $X$ and $Y$ and a graph $G$, let
\begin{align*}
    P(X,Y):= |\{\{x,y\}: x\neq y, x\in X, y\in Y\}|, \quad e(X,Y):=  |\{\{x,y\}\in E(G): x\in X, y\in Y\}|.
\end{align*}
\begin{definition}
    \label{propertyP}
    We say a graph $G$ satisfies property $\ccP$ if for every pair of subsets $(X,Y)$ such that $|X|\geq \sqrt{t}$ and $|Y|\geq \sqrt{t}$,
    \begin{align*}
        \left|\frac{e(X,Y)}{P(X,Y)}-\frac{1}{2}\right| < t^{-1/8}. 
    \end{align*}
\end{definition}
We give a proof of the following Lemma in the appendix for completeness. 
\begin{lemma}
\label{lem: Gnpisregular}
    For $t\geq 2^{64}$, 
    \begin{align*}
        \PP(G(t,1/2) \text{ satisfies } \ccP) \geq 3/4. 
    \end{align*}
\end{lemma}
We observe that $e(X,Y)/P(X,Y)=d(X,Y)$ for a disjoint pair of vertex subsets $X,Y$ and for $X=Y$, we have $e(X,X)/P(X,X)=d(X)$. Thus, if a graph $G$ satisfies property $\ccP$, then for every pair of disjoint subsets $(X,Y)$ such that $|X|\geq \sqrt{t}$ and $|Y|\geq \sqrt{t}$,
    \begin{align}
    \label{eqn: epsregGNP2}
        \left|d(X,Y) - \frac{1}{2}\right|< t^{-1/8}. 
    \end{align}
     Further for every $X\subseteq V(G)$ such that $|X|\geq \sqrt{t}$,
    \begin{align}
    \label{eqn: epsregGNP1}
        \left|d(X)-\frac{1}{2}\right| < t^{-1/8}. 
    \end{align}

\subsection{Ramsey Property of $G(m,1/2)$}
Here we will show that any graph $H$ satisfying $\ccP$ has the property that $H\strongarrow (K_n)^{\Delta}_2$, and consequently the random graph $G(m,1/2)\strongarrow (K_n)^{\Delta}_2$ with high probability.
\begin{lemma}
\label{lem: PropPRamsey}
    For every integer $n\geq 1$, if $H$ is a graph on $m=2^{2^{4n}}$ vertices satisfying property $\ccP$, then,
    \begin{align*}
        H \to (K_n)^{\Delta}_2. 
    \end{align*}
\end{lemma}
 The proof of the lemma follows the approach of Erd\H{o}s and Rado used to prove upper bounds on hypergraph Ramsey numbers in 
 \cite{ER52}.     
\begin{proof}
    Let $H$ be a graph satisfying $\ccP$ on $m= 2^{2^{4n}}$ vertices.  Let $\chi:\cK_3(H)\to \{\text{red}, \text{blue}\}$ be a two coloring of the triangles of $H$. Let 
    \begin{align*}
    \ell = R_2(n-1)+1 < 4^{n}, \quad m = 2^{2^{4n}} \text{ and thus, } \ell\leq \sqrt{\frac{1}{2}\log m}.
    \end{align*} 
    Before we begin, we note that every graph contains a vertex whose degree is at least the average degree. Given a vertex $v\in V(H)$ and $U\subseteq V(H)$, let $d_U(v):= |N(v)\cap U|$. Since $H$ satisfies $\ccP$ (thus $t=m$ in (\ref{eqn: epsregGNP1})), it satisfies the following property:  
     \\[0.3cm]
     For every $X\subseteq V(H)$ such that $|X|\geq \sqrt{m}$, there exists a $v\in X$ such that, 
    \begin{align}
    \label{eqn: avgdeg}
        d_X(v)\geq \left(\frac{1}{2} - m^{-1/8} \right)(|X|-1)= \left(\frac{1}{2} - m^{-1/8} \right)|X| \geq \frac{1}{4}|X|.
    \end{align}
     We will now choose the vertices of a clique $x_1,\dots, x_\ell$ in $V(H)$ with the property that the color of the triangle $\{x_i,x_j,x_k\}$ is determined by the first pair. In other words, for all $1\leq i<j\leq \ell$, the value of $\chi(\{x_i,x_j,x_k\})$ is the same for every $x_k$ with $k>j$. We choose such a clique inductively. \\[0.3cm] 
    In view of (\ref{eqn: avgdeg}), there exists $v\in V(H)$ such that $d(v)\geq m/4$. Let $x_1 = v$ be the first vertex and let $S_1 = N(x_1)$ be the set from which subsequent vertices are chosen. We now describe the inductive step. After $r<\ell$ steps, let us assume we have selected a clique $\{x_1,\dots, x_r\}\subseteq V(H)$ and a set $S_r\subseteq V(H)$ from where $\{x_{r+1},\dots, x_\ell\}$ will be chosen, with the following properties:
        \begin{enumerate}
            \item\label{ER1} $S_r\subseteq N(x_1)\cap \cdots \cap N(x_r)$ and  $|S_r|$ is at least $\left(\frac{1}{2}\right)^{r^2}m$.
            \item\label{ER2} For every $1\leq i<j\leq r$, $\chi(\{x_i, x_j, v\})$ takes the same value for all $v$ in $\{x_{j+1},\dots, x_r\}\cup S_r$. 
           
        \end{enumerate}
    Since $r<\ell < \sqrt{\frac{1}{2}\log m}$,
    \begin{align*}
        |S_r| \geq \left(\frac{1}{2}\right)^{r^2}m \geq \left(\frac{1}{2}\right)^{t^2} m \geq m^{1/2}. 
    \end{align*}
    We will now choose $x_{r+1}$ and $S_{r+1}$. In view of (\ref{eqn: avgdeg}), there exists a vertex $v\in S_r$ such that $d_{S_r}(v)\geq |S_r|/4.$ Let $x_{r+1} = v$ and $S'_{r} = N(x_{r+1})\cap S_r$. For every vertex $v$ in $S'_r$, there are $2^r$ possibilities for the sequence of colors,
    \begin{align*}
        c(v) := (\chi(\{x_1, x_{r+1}, v\}), \chi(\{x_2, x_{r+1}, v\}), \cdots, \chi(\{x_r, x_{r+1}, v\}))\in \{\text{red}, \text{blue}\}^r.
    \end{align*}
    Thus there exists a subset of $S$ of  $S'_r$ with at least $|S_r'|/2^r$ elements such that for every $v\in S$, the sequence of colors $c(v)$ takes the same value $(c_1,\dots, c_r)$. Let $S_{r+1}=S$. In view of \ref{ER1}, we have,
     $$S_{r+1} \subseteq S_r\cap N(x_{r+1})\subseteq N(x_1)\dots \cap N(x_{r+1}),$$
     and since $|S_r'|\geq |S_r|/4$ and $|S_r|\geq \left(\frac{1}{2}\right)^{r^2}m$,
    \begin{align*}
        |S_{r+1}|\geq \left(\frac{1}{2}\right)^r |S_r'|\geq \left(\frac{1}{2}\right)^r \frac{|S_r|}{4}\geq   \left(\frac{1}{2}\right)^r \left(\frac{1}{4}\right)\left(\frac{1}{2}\right)^{r^2}m\geq \left(\frac{1}{2}\right)^{(r+1)^2}m.
    \end{align*}
    Further, by our choice of $S_{r+1}$, for every $1\leq i\leq r$ and $v\in S$, the color of the triangle $\{x_i, x_{r+1}, v\}$ is $c_i$. Consequently, in view of \ref{ER2}, for every $1\leq i<j\leq r+1$, $\chi(\{x_i, x_j, v\})$ takes the same value for all $v$ in $\{x_{j+1},\dots, x_{r+1}\}\cup S_{r+1}$. 
    \\[0.3cm]
    Thus we have a clique $\{x_1,\dots, x_\ell\}$ in $V(H)$ with the property that the color of the triangle $\{x_i,x_j,x_k\}$ is determined by the first pair. We color the pairs of the clique $\{x_1,\dots, x_{t-1}\}$ with the colors $\{\text{red}, \text{blue}\}$. For all $1\leq i<j\leq \ell-1$, the pair $\{x_i, x_j\}$ receives the color $\chi(\{x_i, x_j, x_{j+1}\})$, i.e., the color of the triangles ``looking ahead''. Since $K_{\ell-1}\to (K_{n-1})_2$, there exists a clique $K$ on $n-1$ vertices in the clique $\{x_1,\dots, x_{\ell-1}\}$ such that $K\cup \{x_\ell\}$ forms a $K_n$, all of whose triangles are monochromatic. 
\end{proof}
In view of Lemma \ref{lem: Gnpisregular}, we have the following corollary.
\begin{cor}
\label{cor: GnpRamsey}
     For every integer $n\geq 2$ and $m\geq 2^{2^{4n}}$, 
    \begin{align*}
        \PP(G(m,1/2) \to (K_n)^{\Delta}_2) \geq 3/4.
    \end{align*}
\end{cor} 
\subsection{$G(N,1/2)$ has many monochromatic cliques} Now we are ready to prove Lemma \ref{lem: RamseyGraph}. We remind the reader that given a positive integer $n$, we have,
\begin{align}
\label{eqn: ParametersRamseyGrahp}
        N(n) = N = 2^{2^{2^{40n}}}, \quad m(n) = m = 2^{2^{4n}}, \text{ and further let } \eps(n) = \eps =   \frac{1}{2^{2^{16n}}}.
\end{align}
Further, we will need the following definition. 
\begin{definition}[$(\eps,d)$-regular pair]
    \label{def: regularity}
    Given $0 < \eps < 1$, $0\leq d\leq 1$, a graph $G$, and a pair of disjoint subsets $X,Y\subseteq V(G)$, we say the pair $(X,Y)$ is $(\eps,d)$-regular with respect to $G$ if for every $A\subseteq X$ and $B\subseteq Y$ such that $|A|\geq \eps|X|$ and $|B|\geq \eps |Y|$, 
    \begin{align*}
        \left|d(A,B)-d\right| < \eps.
    \end{align*}
\end{definition}
The graph $G$ described in the following lemma will be the Ramsey graph $G$, i.e., the graph $G$ such that $G\strongarrow (F)_2^{\Delta}$. 
\begin{lemma} 
\label{lem: RamseyGraph}
For every $n\geq 2$, there exists a graph $G$ on $N$ vertices with the following properties:
     \begin{enumerate}[label=\Alabel]
         \item\label{prop: ramseymany} For every $\chi:\cK_3(G)\to \{\text{red}, \text{blue}\}$, $G$ contains at least $\frac{1}{2}\left(\frac{1}{m}\right)^n N^n$ copies of cliques $K_n$ with all triangles of the same color. 
         \item\label{prop: epsregGnp} For every pair of disjoint subsets $X,Y\subseteq V(G)$, such that each of them have size at least $N^{0.9}$, the pair $(X,Y)$ is $(\eps, 1/2)$-regular with respect to $G$. 
     \end{enumerate}
\end{lemma}

\begin{proof}
    Given a graph $G$ on the vertex set $[N]$, let 
    \begin{align*}
        \cR(G):=\{U\in [N]^{(m)}: G[U]\to (K_n)^{\Delta}_2\} \quad \text{ and }\quad  \cB(G):=[N]^{(m)}\setminus \cR(G).  
    \end{align*}
    Informally speaking, we will prove the lemma by showing that there exists an instance of $G(N,1/2)$ for which many subsets of size $m$ are in $\cR(G)$. Further, in view of Lemma \ref{lem: Gnpisregular} and equation (\ref{eqn: epsregGNP2}) (applied with $t=N$), $G(N,1/2)$ will also satisfy \ref{prop: epsregGnp} with probability at least $3/4$. 
    \\[0.3cm]
    Consider $G(N,1/2)$ on the vertex set $[N]$. Since $n\geq 2$, using Corollary \ref{cor: GnpRamsey}, we have that for any fixed $U\in [N]^{(m)}$, 
    \begin{align*}
        \PP( U \in \cB(G(N,1/2))) = 1-\PP(G(m,1/2)\to (K_n)^{\Delta}_2) < 1/4. 
    \end{align*}
    Consequently, summing over all possible choices of $U\in [N]^{(m)}$,
    \begin{align*}
        \EE[|\cB(G(N,1/2))|] < \frac{1}{4}{\binom{N}{m}}. 
    \end{align*}
    In view of Markov's inequality, 
\begin{align*}
    \PP\left(|\cB(G(N,1/2))|\geq \frac{1}{2}{\binom{N}{m}}\right)< 1/2.
\end{align*}
Consequently,
\begin{align*}
    p_A := \PP\left(|\cR(G(N,1/2))|\geq \frac{1}{2}{\binom{N}{m}}\right)\geq 1/2.
\end{align*}
In view of Lemma \ref{lem: Gnpisregular} applied with $t=N$, we have that,
\begin{align*}
    p_B :=\PP\left(G(N,1/2) \text{ satisfies } \ccP \right) \geq 3/4. 
\end{align*}
The probability that $G(N,1/2)$ satisfies 
\begin{align*}
    |\cR(G(N,1/2))|\geq \frac{1}{2}{\binom{N}{m}} \quad \text{and } G(N,1/2) \text{ satisfies } \ccP, 
\end{align*}
is equal to
$$1-(1-p_A)-(1-p_B)= \frac{1}{4}>0.$$ 
Thus there exists a $G$ such that 
\begin{align}
\label{eqn: propRGpre}
    |\cR(G)|\geq \frac{1}{2}{\binom{N}{m}} \quad \text{and } G \text{ satisfies } \ccP. 
\end{align}
We will now show that (\ref{eqn: propRGpre}) implies that $G$ satisfies \ref{prop: ramseymany} and \ref{prop: epsregGnp}. First we show that $G$ satisfies \ref{prop: ramseymany}. 
\\[0.3cm]
Let $\chi:\cK_3(G)\to \{\text{red}, \text{blue}\}$ be a coloring of the triangles. Every $U\in \cR(G)$ must contain a copy of $K_n$ such that all the triangles are of the same color. Further, for any $K_n$ in $G$, there are at most ${\binom{N-n}{m-n}}$ subsets of $V(G)$ of size $m$ that contain it. Consequently, the number of $K_n$ in $G$  with all triangles of the same color, is at least, 
\begin{align}
\label{eqn: cliquesHn}
    |\cR(G)|\cdot \frac{1}{{\binom{N-n}{m-n}}} \geq \frac{1}{2}\frac{{\binom{N}{m}}}{{\binom{N-n}{m-n}}}\geq \frac{1}{2m^n}N^n. 
\end{align}
Now we show that $G$ satisfies \ref{prop: epsregGnp}. Let $X$, $Y$ be a pair of disjoint subsets in $V(G)$, each of size at least $N^{0.9}$. We will show that $(X,Y)$ is $(\eps,1/2)$-regular. Let $A\subseteq X$ and $B\subseteq Y$ such that,
\begin{align*}
    |A|\geq \eps |X|\geq \eps N^{0.9} \quad \text{ and } \quad |B|\geq \eps |Y|\geq \eps N^{0.9}. 
\end{align*}
In view of the choice of $\eps$ and $N$ in terms of $n$ in (\ref{eqn: ParametersRamseyGrahp}), we have that,
\begin{align*}
    \eps \geq (\log N)^{-1}\quad \text{ and hence } \eps N^{0.9} \geq \sqrt{N}.
\end{align*}
Since $G$ satisfies $\ccP$, by equation (\ref{eqn: epsregGNP2}) with $t=N$, we have,
\begin{align*}
    \left|d_G(A,B)-\frac{1}{2}\right| <N^{-1/8} < \eps.
\end{align*}
Hence, $G$ satisfies \ref{prop: epsregGnp}. 

\end{proof}

\section{Proof of Theorem \ref{thm: main}}\label{sec:main}
In order to prove Theorem \ref{thm: main}, given a graph $F$ on $n$ vertices such that $\cK_3(F)$ is linear, we need to find a graph $G$ on $N= 2^{2^{2^{40n}}}$ vertices such that for every 2-coloring of $\cK_3(G)$, there is an induced copy of $F$ with all its triangles of the same color. We will show that the graph $G$ in Lemma \ref{lem: RamseyGraph} is such a graph. However, before we proceed, we will need to state a few definitions and Lemmas that we use in the proof.

\subsection{$(\eps,d)$-dense tuples in $r$-uniform hypergraphs} 
As mentioned before, once we obtain the graph $G$ in Lemma \ref{lem: RamseyGraph}, along with a two coloring of $\cK_3(G)$, we ``weakly regularise'' the monochromatic (say blue) cliques. Here we  define the concept of ``weak regularity''. 
\begin{definition}[$(\eps,d)$-dense tuples]
\label{def: dense}
Given $0 < \eps < 1$, $0< d\leq 1$, a $r$-uniform hypergraph $H$, a pairwise disjoint subsets of vertices $V_1,\dots, V_r$, we say that $(V_1,\dots, V_r)$ is $(\eps,d)$-dense with respect to $H$ if for every $W_1,\dots, W_r$ such that $W_i\subseteq V_i$ and $|W_i|\geq \eps |V_i|$, 
\begin{align*}
    d(W_1,\dots, W_r) \geq d.
\end{align*}
\end{definition}
We state the following lemma, which is an extension of a result in \cite{PR02} for graphs to hypergraphs of higher uniformity. Since the proof follows similar lines to the arguments in \cite{PR02}, we state it now and postpone the proof to the Appendix. 
\begin{lemma}
\label{lem: pengrucinski}
    Given $\eps > 0$ and $d> 0$ and an $n$-partite $n$-uniform hypergraph $H$ on $(V_1,\dots, V_n)$ such that $d_H(V_1,\dots, V_n)\geq d$, there exist subsets $U_i\subseteq V_i$ for $i\in [n]$, such that 
    $$|U_i|\geq d^{\frac{4}{\eps^n}\ln(\frac{1}{\eps})}|V_i|,$$
    for every $i\in [n]$, and $H[U_1,\dots, U_n]$ is $(\eps,d/2)$-dense. 
\end{lemma}
\subsection{Embedding Lemma} Here we state the following embedding lemma for embedding linear 3-uniform hypergraphs. The following Lemma is an extension of a result in \cite{KNR10}. Since the proof also follows similar lines to the arguments in \cite{KNR10}, we state it now and postpone the proof to the Appendix.
\begin{lemma}[Embedding Lemma]
\label{lem: embedding}
Given a graph $F$ such that $\cK_3(F)$ is linear, on $n$ vertices, and $0<d\leq 1$, $0\leq \eps < d^{n^2}$, the following holds. \\
For all $n$-partite graphs $G$ and $n$-partite 3-uniform hypergraphs $H$ such that $E(H)\subseteq \cK_3(G)$ with vertex partition $(V_1, \dots, V_n)$ such that
\begin{enumerate}[label =\nlabel]
    \item\label{prop: em1} For all $1\leq i < j\leq n$, the pairs $(V_i,V_j)$ are $(\eps,1/2)$-regular in $G$. 
    \item \label{prop: em2} For all $1\leq i < j < k\leq n$ are $(\eps, d)$-dense in $H$, 
\end{enumerate}
there exists an induced copy of $F$ in $G$ such that every triangle of $F$ is an edge of $H$.
\end{lemma}
Now we are able to prove Theorem \ref{thm: main}. 
\begin{proof}[Proof of Theorem \ref{thm: main}]
 Let $F$ be a graph on $n\geq 2$ vertices such that $\cK_3(F)$ is linear, let:
\begin{align}
\label{eqn: ParametersRamseyGrahp1}
        N(n) = N = 2^{2^{2^{40n}}}, \quad m(n) = m = 2^{2^{4n}}, \quad \eps(n) = \eps =   \frac{1}{2^{2^{16n}}}, \quad d(n) = d = \frac{1}{2^{2^{32n}}}.
\end{align}
     In view of Lemma \ref{lem: RamseyGraph}, there exists a graph $G$ on $N$ vertices satisfying \ref{prop: ramseymany} and \ref{prop: epsregGnp}. Let $\chi:\cK_3(G)\to \{\text{red}, \text{blue}\}$ be a coloring of its triangles. 
    \\[0.3cm]
    Since $G$ satisfies Property \ref{prop: ramseymany}, there are at least $\frac{1}{2m^n}N^n$ monochromatic copies of $K_n^{(3)}$ in $\cK_3(G)$. Let us assume that a majority of them are blue. Let $H_n$ be the $n$-uniform hypergraph on vertex set $V(G)$, with edges being the vertex sets of the blue $K_n^{(3)}$s. We have, 
    \begin{align}
    \label{eqn: cliquesHn2}
        |E(H_n)|\geq \frac{1}{2}\cdot\frac{1}{2m^n}N^n = \frac{1}{4m^n}N^n.
    \end{align}
    Before we proceed with steps \ref{II} and \ref{III} mentioned in the outline, we will partition the vertex set of $V(G)$ into $n$ parts so that all the parts are of the same size and there are many crossing cliques in $H_n$ with respect to the partition. 
    \\[0.3cm]
    More formally, let $V(G)=V_1\cup\ldots\cup V_n$ be a partition chosen uniformly out of all the equitable partitions of $V(G)$, i.e., the family of partitions such that $|V_1|=\ldots=|V_n|=N/n$. Note that although $N/n$ might not be an integer, we prefer to write in this way, since it simplify the exposition and has no significant effect on the arguments.
    \\[0.3cm]
    Given an edge $e \in H_n$, the probability that it is in $H_n[V_1,\ldots,V_n]$ is $n!/n^n$. Therefore, by (\ref{eqn: cliquesHn2}), the expected number of edges in $H_n[V_1,\ldots, V_n]$ is
    \begin{align*}
        \EE(e_{H_n}(V_1,\ldots,V_n))=\frac{n!}{n^n}\cdot \frac{1}{4m^n}N^n \geq \frac{1}{4m^n}\left(\frac{N}{n}\right)^n
    \end{align*}
    and consequently there exists a partition $(V_1,\ldots,V_n)$ with
    \begin{align}
        \label{eqn: densityHn}
        d_{H_n}(V_1,\dots, V_n) \geq \frac{1}{4m^n}\left(\frac{N}{n}\right)^n\cdot \frac{1}{|V_1|\cdots |V_n|} \geq \frac{1}{4m^n} > 2d. 
    \end{align}
Now we are ready to proceed with steps \ref{II} and \ref{III} mentioned in the proof outline, i.e., we will ``weakly regularise'' $H_n$ and then embed the copies of $F$. First we regularise. 
    \\[0.3cm]
    Applying Lemma \ref{lem: pengrucinski} to $H_n[V_1,\dots, V_n]$ with $\eps = \eps(n)$ in (\ref{eqn: ParametersRamseyGrahp1}) and density $2d$ (as obtained in (\ref{eqn: densityHn})), we obtain subsets $U_i\subseteq V_i$ for each $i\in [n]$ such that, 
    \begin{align}
    \label{eqn: sizeofUi2}
        |U_i| \geq (2d)^{\left(\frac{4}{\eps^n}\right)\ln(\frac{1}{\eps})}|V_i| 
    \end{align}
    and $H_n[U_1,\dots, U_n]$ is $(\eps, d)$-dense. We now compute the size of $U_i$ in terms of $N$. In view of the relationship between $\eps,d, N$ in terms of $n$ in (\ref{eqn: ParametersRamseyGrahp1}), we have,
    \begin{align*}
        n< \log\log\log N, \quad 2d\geq (\log N)^{-1}. 
    \end{align*}
    Further, the exponent in the right hand side of (\ref{eqn: sizeofUi2}) satisfies,
    \begin{align*}
         \frac{4}{\eps^n}\ln\left(\frac{1}{\eps}\right) = 4\cdot 2^{2^{16n}n}\cdot 2^{16n}\ln 2 < 2^{2^{32n}} < \sqrt{\log N}.
    \end{align*}
    And consequently, in view of (\ref{eqn: sizeofUi2}) together with the fact that $|V_i|=N/n$, for sufficiently large $N$, 
    \begin{align*}
        |U_i|\geq (2d)^{\left(\frac{4}{\eps^n}\right)\ln(\frac{1}{\eps})}\left(\frac{N}{n}\right) \geq\frac{N}{(\log N)^{\sqrt{\log N}}\log\log\log N}> N^{0.9}.
    \end{align*}
    To summarise, we obtain vertex sets $U_i$, each of size at least $N^{0.9}$ such that the $n$-uniform hypergraph of blue cliques $H_n[U_1,\dots, U_n]$ is $(\eps,d)$-dense. We will now proceed with the set up for the application of Lemma \ref{lem: embedding}.
    \\[0.3cm]
    Let $\Tilde{G}:= G[U_1,\dots,U_n]$ be the $n$-partite graph induced by $G$ on $U_1,\dots, U_n$. Let $BT = BT^{(3)}(H_n)$ be the $n$-partite 3-uniform hypergraph on $U_1, \dots, U_n$ whose edges are the blue triangles of the cliques in $H_n$. More formally, the triple $\{v_1,v_2,v_3\}$ is an edge in $BT$ iff there exists a clique in $H_n[U_1,\dots, U_n]$ that contains it. We will now show that the graphs, $\Tilde{G}$ and ``the blue triple system'' $BT$ satisfy hypotheses \ref{prop: em1} and \ref{prop: em2} of Lemma \ref{lem: embedding} respectively.
    \\[0.3cm]
    Since $|U_i|\geq N^{0.9}$ for every $i\in[n]$, by Property \ref{prop: epsregGnp} of $G$ in Lemma \ref{lem: RamseyGraph}, we have that for every $1\leq i<j\leq n$, the pair $(U_i,U_j)$ is $(\eps, 1/2)$-regular with respect to $G$. Hence $\Tilde{G}$ satisfies \ref{prop: em1}. 
    \\[0.3cm]
    We now show that $BT$ satisfies \ref{prop: em2}, i.e., for every $1\leq i<j<k\leq n$, the 3-partite 3-uniform hypergraph $BT[U_i,U_j,U_k]$ is also $(\eps,d)$-dense. Suppose without loss of generality that $\{i,j,k\}=\{1,2,3\}$ and let $W_i\subseteq U_i$ with $|W_i|\geq \eps|U_i|$, for $i\in\{1,2,3\}$. For $4\leq r\leq n$, let $W_r = U_r$ . Since $H_n[U_1,\dots, U_n]$ is $(\eps,d)$-dense, we must have, 
    \begin{align}
    \label{eqn: trianglesub}
        e_{H_n}(W_1,\dots, W_n)\geq d|W_1|\cdots |W_n|.
    \end{align}
    However each triple in $(W_1, W_2, W_3)$ can be in at most, $|W_4|\cdots|W_n|$ edges of $H_n$. Thus, we have,
    \begin{align*}
        d|W_1|\cdots |W_n| \leq e_{H_n}(W_1,\dots, W_n) \leq d_{BT}(W_1,W_2,W_3)|W_1||W_2||W_3| \prod_{r\geq 4}|W_r|,
    \end{align*}
    which implies, 
    \begin{align*}
        d_{BT}(W_1,W_2,W_3)\geq d,
    \end{align*}
    and consequently $BT$ satisfies \ref{prop: em2}. 
    \\[0.3cm]
    Now we are ready to proceed with the final step \ref{III}. Applying Lemma \ref{lem: embedding}, with $F$ as the graph in the statement of Theorem \ref{thm: main}, $\eps, d$ as in (\ref{eqn: ParametersRamseyGrahp1}) (note that $\eps<d^{n^2})$ and $G$ and $BT$ as the graph and hypergraph in \ref{prop: em1} and \ref{prop: em2}, we have that there exists an induced copy of $F$ in $G[U_1,\dots, U_n]$ such that its triangles are in $BT$. Since every triangle of $BT$ is a triangle of some clique in $H_n$, they must all be of blue color, and this implies $G\strongarrow (F)_2^\Delta$.
\end{proof}

\section{Proof of Theorem \ref{thm:bipartite}}\label{sec:bipartite}

Before proving Theorem \ref{thm:bipartite}, we will state a well-known result about extremal numbers of complete $k$-partite $k$-graphs. Given a $k$-graph $F$, The \emph{extremal number $\ex(N,F)$} is the largest number of edges in a $k$-graph $G$ on $N$ vertices without a copy of $F$. The \emph{complete balanced $k$-partite $k$-graph} $K_{t,\ldots,t}^{(k)}$ is the $k$-graph whose vertex set is the disjoint union $V=V_1\cup\ldots \cup V_k$ where $|V_i|=t$ for $1\leq i \leq k$ and whose edges are all the transversal $k$-tuples, i.e., 
\begin{align*}
    E(K_{t,\ldots,t}^{(k)})=\{\{x_1,\ldots,x_k\}:\: x_i\in V_i \}
\end{align*}
Extending a result by K\"{o}vari, S\'{o}s and Tur\'{a}n \cite{KST54}, Erd\H{o}s showed that the extremal number of $K_{t,\ldots,t}^{(k)}$ is $o(n^k)$.

\begin{thm}[\cite{E64}]\label{thm:erdos64}
Let $N, k, t\geq 1$ be integers. There exists a positive constant $C:=C(k,t)$ such that
\begin{align*}
    \ex(N,K_{t,\ldots,t}^{(k)})\leq CN^{k-1/t^{k-1}}.
\end{align*}
\end{thm}

We are now able to prove the theorem

\begin{proof}[Proof of Theorem \ref{thm:bipartite}]
Let $F\in \cB_n$ be a graph on $n$ vertices. By definition, there exists a partition $V(F)=A\cup B$ of the vertices such that $F[A]$ is a triangle free graph and $B$ is independent. Let $k=|A|$ and $s=|B|$ and write $A=\{a_1,\ldots,a_k\}$ and $B=\{b_1,\ldots, b_s\}$. For $1\leq i \leq s$, let $A_i=N(b_i)\subseteq A$ be the neighborhood of $b_i$ in $A$. We may assume that all the sets $A_i$ are distinct. Indeed, if this is not the case, we can enlarge the neighborhood of each $b_i$ by adding a unique new vertex to each $A_i$. A monochromatic induced copy of the enlarged graph $F$ will always contain a monochromatic induced copy of the original $F$. 
\\[0.3cm]
Let $H$ be a graph on vertex set $[m]$ such that $H\strongarrow F[A]$. By equation (\ref{eq:inducedgraphs}), we can choose $H$ such that 
\begin{align}\label{eq:m}
    m\leq 2^{c|A|\log |A|}\leq 2^{cn\log n}
\end{align}
for positive constant $c$. We now describe the construction of the host graph $G$. Set $N=m^{k2^k}$. Let $V_1, \ldots, V_m$ be disjoint copies of $[N]$ and $U$ be a disjoint copy of $[N]^m$. We define the  vertex set of $G$ by
\begin{align*}
    V(G)=\left(\bigcup_{i=1}^m V_i\right) \cup U. 
\end{align*}
The edge set is described as the union of two disjoint set of edges $E(G)=E_1\cup E_2$. The first set 
\begin{align*}
    E_1=\left\{\{x_i,x_j\} \in V_i\cup V_j:\: \{i,j\} \in E(H) \text{ and } x_i\in V_i, x_j \in V_j\right\}
\end{align*}
is a blow-up of the graph $H$, where we replace each vertex $i$, by the independent set $V_i$. The second set is given by 
\begin{align*}
    E_2=\bigcup_{i=1}^m\left\{\{x,(x_1,\ldots,x_m)\}\in V_i\times U:\: x_i=x \right\}.
\end{align*}
That is, for every transversal copy $H^*$ of $H$ with vertex set $V(H^*)=\{x_1,\ldots,x_m\}$ there is a vertex $(x_1,\ldots,x_m) \in U$ whose neighborhood in $G$ is exactly $V(H^*)$. Note that by (\ref{eq:m}) and our choice of $N$, we have
\begin{align*}
    v(G)\leq N^m+mN\leq m^{mk2^k}+m^{k2^k+1} \leq 2^{2^{cn\log n}}
\end{align*}
for a positive constant $c$.
\\[0.3cm]
It remains to prove that $G\strongarrow (F)_2^{\Delta}$. Let $\chi: \cK_3(G)\rightarrow \{\text{red},\text{blue}\}$ be a $2$-coloring of the triangles of $G$. Let $\cT_H$ be the set of all transversal copies of $H$ in $V_1\cup\ldots\cup V_m$. For convenience, we will denote the elements of $\cT_H$ by their vertex set.For $T\in \cT_H$ with $T=\{x_1,\ldots, x_m\}$, we define an auxiliary coloring of its edges $\phi_T: E(G[T]) \rightarrow \{\text{red}, text{blue}\}$ by
\begin{align*}
    \phi_T(x_i, x_j)=\chi(\{x_i, x_j, (x_1,\ldots, x_m)\}),
\end{align*}
where $\{x_i,x_j\} \in E(G[T])$ and $(x_1,\ldots,x_m) \in U$. That is, the coloring $\phi_T$ assigns to each edge $\{x_i,x_j\} \in E(G(T))$ the color of the triangle $\{x_i,x_j,(x_1,\ldots,x_m)\}$, where $(x_1,\ldots,x_m)$ is the vertex in $U$ whose neighborhood is $T$. Our choice of $H$ gives that there is an induced monochromatic copy $F^T[A]\subseteq G[T]$ of $F[A]$ with respect to $\phi_T$. Let $I_T \in [m]^{(k)}$ be the set of indices $i\in [m]$ such that $V_i\cap V(F^T[A]) \neq \emptyset$. By the pigeonhole principle, we may assume that there exists a set of indices $I\in [m]^{(k)}$ and subset $\cT\subseteq \cT_H$ of size $N^m/2\binom{m}{k}$ such that $I_T=I$ and $F^T[A]$ is of color blue for every $T\in \cT$.
\\[0.3cm]
For each $F^*[A]$ transversal copy of $F[A]$ in $G[\bigcup_{i\in I}V_i]$, there is at most $N^{m-k}$ graphs $T\in \cT$ such that $F^T[A]=F^*[A]$. Therefore, by our choice of $N$, the number of distinct blue copies of $F[A]$ in $G[\bigcup_{i\in I}V_i]$ is at least
\begin{align*}
\frac{N^m}{N^{m-k}2\binom{m}{k}}\geq \frac{1}{2}\left(\frac{N}{m}\right)^k\gg N^{k-\frac{1}{2^{k-1}}}.
\end{align*}
By viewing each of these graphs as a edge of a $k$-partite $k$-graph with vertex set $\bigcup_{i\in I}V_i$, we can apply Theorem \ref{thm:erdos64} for $t=2$ to obtain a blue blow-up of $F[A]$ in the vertex set $\bigcup_{i\in I}V_I'$, where $V_i'\subseteq V_i$ with $|V_i|=2$.
\\[0.3cm]
Suppose without loss of generality that $I=[k]$. Write $V_i'=\{x_i,y_i\}$ for $1\leq i \leq k$ and denote each copy of $F[A]$ in the blowup by its set of vertices. For each copy $X$, there exists at least one $T\in \cT$ such that $X\subseteq T$. Let $u_X$ be the vertex in $U$ whose neighborhood is exactly $T$. In particular, all the triangles in $G[X\cup \{u_X\}]$ are colored blue. Moreover, if $X\neq X'$, then the vertices $u_X$ and $u_{X'}$ are distinct. We will now describe how to embed a copy of $F$ in $G[\left(\bigcup_{i=1}^k \{x_i,y_i\}\right)\cup U]$. Let $X_0=\{x_1,\ldots,x_m\}$ and for $1\leq i \leq b$, let $X_i$ be the $k$-set defined by
\begin{align*}
X_i=\{x_j:\: a_j\in A_i\} \cup \{y_j:\: a_j\notin A_i\}. 
\end{align*}
Note that since all the sets $A_i$ are disjoint, we have that all the sets $X_i$ are distinct. The embedding $\psi: V(F)\rightarrow \left(\bigcup_{i=1}^k\{x_i,y_i\}\right)\cup U$ can be given as follows: For each $1\leq i \leq k$, let $\psi(a_i)=x_i$. This embeds $F[A]$ into the copy $G[X_0]$. For each $1\leq i \leq s$, let $\psi(b_i)=u_{X_i}$. This embed the vertex $b_i$ into a vertex whose neighborhood in $X_0$ is precisely $\psi(A_i)=X_0\cap X_i$. Moreover, all the triangles in $\psi(A\cup B)$ are blue. Therefore, $\psi(A\cup B)$ is a blue induced copy of $F$, which finishes the proof.
\end{proof}

\section{Embedding tight trees}\label{sec:tree}
In this section, we prove Proposition \ref{prop:tree}. We note that Beck \cite{beck90} established a polynomial upper bound for the induced size Ramsey of a tree. Our approach here is similar. For a graph $G$ and a $3$-graph $\cH$, let
\begin{align*}
    d_{\cH}(\{x,y\})=|\left\{z \in V(G):\: \{x,y,z\} \in \cH\right\}|
\end{align*}
be the \emph{degree of $\{x,y\}$ in $\cH$}. The \emph{average degree of $\ovl{d_{\cH}}:=\ovl{d_{\cH}}(G)$ of the graph $G$ in $\cH$} is given by
\begin{align*}
    \ovl{d_{\cH}}=\frac{1}{e(G)}\sum_{\{x,y\}\in G}d_{\cH}(\{x,y\})
\end{align*}
For a set $S\subseteq V(G)$, let 
\begin{align*}
    N(S)&=\{v\in V(G):\: \{v,x\}\in G \text{ for every }x\in S\}\\
    \Gamma(S)&=\{v\in V(G):\: \{v,x\}\in G \text{ for some }x\in S\}
\end{align*}
be the \emph{common} and \emph{joint neighborhood} of $S$, respectively. The next lemma is the key component of the proof.

\begin{lemma}\label{lem:beck}
    Let $G$ be a graph with average triangle degree $d=\ovl{d_{\cK_3(G)}}$ and with the following property: For every $|S|\leq n$ and edge $\{x,y\} \in G$ it holds that
    \begin{align*}
        \left|N(\{x,y\}) \cap \Gamma(S) \right|\leq 0.01d.
    \end{align*}
    If $F$ is a graph on $n$ vertices such that $\cK_3(F)$ is a tight tree, then $G\strongarrow (F)^{\Delta}$
\end{lemma}

\begin{proof}
Let $\chi: \cK_3(G)\rightarrow \{\text{red}, \text{blue}\}$ be a $2$-coloring of the triangles of $G$. We may assume that the majority of the triangles are colored red. Let $\cR\subseteq \cK_3(G)$ be the set of red triangles. Then it holds that
\begin{align}\label{eq:initial}
    e(\cR)\geq \frac{e(\cK_3(G))}{2}=\frac{d \cdot e(G)}{6}
\end{align}
For a subgraph $G'\subseteq G$, let $\cR(G')=\cR\cap \cK_3(G')$ be the set of red triangles that are triangles of $G'$. Let $\Omega$ be the set of subgraphs $G'\subseteq G$ with the property that 
\begin{align}\label{eq:minimal}
    e(\cR(G'))\geq \frac{d\cdot e(G')}{6}
\end{align}
for every $G'\in \Omega$. Inequality $(\ref{eq:initial})$ implies that $G\in \Omega$ and consequently $\Omega\neq \emptyset$. Let $H \in \Omega$ be the graph with the smallest number of edges in $\Omega$. We claim that for every edge $e\in H$ it holds that
\begin{align}\label{eq:mindeg}
    d_{\cR(H)}(e)\geq \frac{d}{6}.
\end{align}
Suppose to the contrary that there exists edge $e \in H$ such that (\ref{eq:mindeg}) does not hold. Then the graph $H'=H\setminus \{e\}$ has
\begin{align*}
e(\cR(H'))=e(\cR(H))-d_{\cR(H)}(e)\geq \frac{d\cdot (e(H)-1)}{6}=\frac{d\cdot e(H')}{6}
\end{align*}
red triangles, satisfying property (\ref{eq:minimal}). This contradicts the minimality of $H$.
\\[0.3cm]
Let $F$ be a graph on $n$ vertices such that $\cK_3(F)$ is a tight tree. Let $e_1,\ldots,e_{n-2}$ be the ordering of the edges such that for each $i\geq 2$, $e_i$ has a vertex $v_i$ that does not belong to any previous edge and $e_i\setminus \{v_i\}$ is contained in $e_j$ for some $j<i$. We will construct an embedding $\psi:V(F)\rightarrow V(G)$ sequentially such that each triangle of $F$ is colored red. We start by taking any red triangle in $H$ and embedding the first three vertices of $F$ in any arbitrary order. Suppose that we already have embedded the triangles $e_1,\ldots, e_i$ and we want to embed $e_{i+1}$. Let $\{x,y\}=e_{i+1}\setminus {v_{i+1}}$ and $S=V\left(\left(\bigcup_{j=1}^i e_i\right)\setminus \{x,y\}\right)$.To embed $v_{i+1}$, it suffices to find a vertex $z$ such that $\{\psi(x),\psi(y),z\}$ forms a red triangle and $z$ is not adjacent to every element in $\psi(S)$, i.e., $z\notin \Gamma(\psi(S))$. Since $|S|\leq n$, by the hypothesis of the statement we have that
\begin{align*}
\left|N(\{\psi(x),\psi(y)\}) \cap \Gamma(\psi(S)) \right|\leq 0.01d.
\end{align*}
Moreover, relation (\ref{eq:mindeg}) says that there are at least $d/6$ red triangles in $H$ sharing the edge $\{\psi(x),\psi(y)\}$. Therefore, there are $d/6-0.01d>0$ choices of vertex $z$ to play the role of $v_{i+1}$ and we can take $\psi(v_{i+1})=z$. This concludes the proof of the lemma. 
\end{proof}

We prove the proposition by finding a suitable graph $G$ satisfying the properties of Lemma \ref{lem:beck}.

\begin{proof}[Proof of Proposition \ref{prop:tree}]
Let $N=n^4$ and $p=\frac{1}{200n}$ and $\epsilon>0$ a sufficiently small real number. First, we observe that whp the random graph $G_{N,p}$ satisfies the following two events:
\begin{enumerate}
    \item[(A)] For every pair of vertices $\{x,y\}$, we have
    \begin{align*}
        |N(\{x,y\})|=(1\pm \epsilon)p^2N.
    \end{align*}
    \item[(B)] For every set $S$ with $|S|\geq n$ and pair $\{x,y\}$ disjoint of $S$, we have
    \begin{align*}
        |N(\{x,y\})\setminus \Gamma(S)|=(1\pm \epsilon)p^2(1-p)^{|S|}N.
    \end{align*}
\end{enumerate}
Note that by Chernoff bound
\begin{align*}
\PP\left(\left||N(\{x,y\}|-p^2N\right|\geq \epsilon p^2N\right)\leq \exp(-c_{\epsilon}p^2N).
\end{align*}
Since, there are at most $N^2$ pairs $\{x,y\}$ a union bound argument shows that the probability that event (A) fails is at most
\begin{align*}
    N^2\exp(-c_{\epsilon}p^2N)\leq n^8\exp(-c'_{\epsilon}n^2)
\end{align*}
Similarly, Chernoff bound and a union bound argument gives that the probability that event (B) fails is at most
\begin{align*}
    \binom{N}{|S|}N^2\exp(-c_{\epsilon}p^2(1-p)^{|S|}N)\leq n^{4(n+2)}\exp(-c'_{\epsilon}n^2).
\end{align*}
Therefore, the probability that one of the events (A) and (B) fails is goes to $0$ as $n\rightarrow \infty$.

Take an instance of $G\sim G_{N,p}$ satisfying events (A) and (B). Then the average triangle degree is given by
\begin{align*}
    d=\ovl{d_{\cK_3(G)}}\geq (1-\epsilon)p^2N
\end{align*}
Moreover, event (A) and (B) gives to us that for set $|S|\leq n$ and pair $\{x,y\}$, we have
\begin{align*}
    |N(\{x,y\})\cap \Gamma(S)|&=|N(\{x,y\})|-|N(\{x,y\})\setminus \Gamma(S)|\\
    &\leq (1+\epsilon)p^2N-(1-\epsilon)p^2(1-p)^{|S|}N\\
    &\leq (1+\epsilon)p^2N-(1-\epsilon)p^2(1-p|S|)N\\
    &\leq (0.005+2\epsilon)p^2N\leq 0.01 d.
\end{align*}
Hence, $G$ satisfies the hypothesis of Lemma \ref{lem:beck} and consequently $G\strongarrow (F)^{\Delta}$ for any $F$ as in the statement. This concludes the proof.
\end{proof}

\section{Concluding remarks}

We would like to remark that the original proofs in \cites{Deuber75, NR75} work for a more general setting where one can color complete graphs of size $k$ instead of triangles. To be more precise, given graphs $G$ and $F$, we say that $G\strongarrow (F)^{K_r}$ if any $2$-coloring of the complete graphs of size $r$ in $G$ yiels a induced copy of $F$ with all $K_r$ of the same colors. Let $R_{\Ind}^{K_r}(F)$ be the smallest number of vertices in a graph $G$ such that $G\strongarrow (F)^{K_r}$. We note that all the proofs in this paper can be extended to $r\geq 3$, however for simplicity we preferred to keep it for triangles. We believe that the numbers $R_{\Ind}^{K_r}(F)$ should grow as the classical Ramsey numbers $R^{(r)}(F)$. Let $t_k$ be the tower function defined by $t_0(x)=x$ and $t_i(x)=2^{t_{i-1}(x)}$.

\begin{conj}
Show that for any graph $F$ on $n$ vertices, we have
\begin{align*}
    R_{\Ind}^{K_r}(F)\leq t_{r-1}(cn)
\end{align*}
for some constant $c:=c(r)$ depending on $r$
\end{conj}

\bibliography{ref}
\appendix\label{sec:appendix}

\section{Proof of Lemma \ref{lem: Gnpisregular}}
\begin{proof}
Let $G(t,1/2)$ be the random graph with vertex set $V$ of size $t$. We first show that for every pair of subsets of size $X$ and $Y$ such that $X$ and $Y$ are disjoint and $|X|= |Y|=\sqrt{t}$, we have,
    \begin{align*}
            \left|e(X,Y) - \frac{1}{2}P(X,Y)\right|< t^{-1/8}P(X,Y). 
    \end{align*}   
    Let $X$ and $Y$ be subsets of $V$ such that $|X|=|Y|=\sqrt{t}$. Then we have, 
    \begin{align*}
        \mu:=\EE[e(X,Y)] = \frac{1}{2}P(X,Y)\geq \frac{1}{2}{\binom{\sqrt{t}}{2}}.
    \end{align*}
    Let $t>2^{64}$ and  $\eps = t^{-1/8}$. By Chernoff Bound, we have, 
    \begin{align*}
        \PP(|e(X,Y)-\mu|>\eps\mu) &< 2\exp\left(-\frac{\eps^2\mu}{3}\right)
        < \exp\left(-\frac{t^{3/4}}{24}\right).
    \end{align*}
    Thus the probability that there exists a pair of subsets $X,Y$ with size $\sqrt{t}$ and, 
    \begin{align*}
        \left|e(X,Y) - \frac{1}{2}P(X,Y)\right|> t^{-1/8}P(X,Y). 
    \end{align*}
    is at most,
    \begin{align*}
        \displaystyle \sum_{\substack{Y\subseteq V\\ |Y|=\sqrt{t}}}\sum_{\substack{X\subseteq V\\ |X|=\sqrt{t}}}  \exp\left(-\frac{t^{3/4}}{24}\right) & =  {\binom{t} {\sqrt{t}}}^2 \exp\left(-\frac{t^{3/4}}{24}\right)< 1/4,
    \end{align*}
    for sufficiently large $t$ (in particular $t>2^{64})$. \\[0.3cm ]
    Consequently, $G(t,1/2)$ satisfies the property that for every pair of subsets of size $X$ and $Y$ such that $|X|= |Y|=\sqrt{t}$, 
    \begin{align*}
          \left|e(X,Y) - \frac{1}{2}P(X,Y)\right|< t^{-1/8}P(X,Y). 
    \end{align*} 
    This implies that given a pair subsets of vertices $X$ and $Y$ such that $|X|\geq \sqrt{t}$ and $|Y|\geq \sqrt{t}$, we must have,
    \begin{align*}
            \left|e(X,Y) - \frac{1}{2}P(X,Y)\right|> t^{-1/8}P(X,Y). 
    \end{align*} 
    since otherwise there exists $X_0\subseteq X$ and $Y_0\subseteq X$ with $|X_0|=|Y_0|=\sqrt{t}$ such that 
     \begin{align*}
           \left|e(X_0,Y_0) - \frac{1}{2}P(X_0,Y_0)\right|> t^{-1/8}P(X_0,Y_0). 
    \end{align*} 
\end{proof}

    \section{Proof of Lemma \ref{lem: pengrucinski}}  Before we formally prove the statement, we informally describe the proof here. The proof is based on a density increment argument.  Let $0<\eps < 1$ and $0< d\leq 1$ and let $H$ be an $r$-partite $r$-uniform hypergraph with pairwise disjoint vertex sets $V_1,\dots, V_r$ such that $d_H(V_1,\dots, V_r)\geq d.$\\[0.3cm]
    At each step $s$, we start with subsets $U_1^{(s)},\dots, U_r^{(s)}$ of $V_1,\dots, V_r$ respectively such that $d_H(U_1^{(s)},\dots, U_r^{(s)}) = d_s$. If $U_i^{(s)},\dots, U_r^{(s)}$ is not $(\eps,d/2)$-dense, then we show that for each $i\in [r]$, there exists  $U_i^{(s+1)}\subseteq U_i^{(s)}$ with $|U_i^{(s+1)}|\geq  \eps|U_i^{(s)}|$ , such that the density of $d_H(U_1^{(s+1)},\dots, U_r^{(s+1)})$ increases. In particular, we will show that, 
    $$d_{s+1} = d_H(U_1^{(s+1)},\dots, U_r^{(s+1)})\geq d_s\left(1+\frac{\eps^r}{2}\right).$$
    Since the density is bounded above by 1, the process must stop after a certain number of steps. We now prove the lemma. 
    \begin{proof}[Proof of Lemma \ref{lem: pengrucinski}]
    Fix $0<\eps < 1$ and $0< d\leq 1$. Let $H$ be an $r$-partite $r$-uniform hypergraph with pairwise disjoint vertex sets $V_1,\dots, V_r$ and let,
    \begin{align*}
        d_H(V_1,\dots, V_r) \geq d. 
    \end{align*}
    Assume that $(V_1,\dots, V_r)$ is not $(\eps,d/2)$-dense with respect to $H$ for otherwise we are done. Then there must exist subsets $U_1,\dots, U_r$ such that $U_i\subseteq V_i$ and $|U_i|\geq \eps |V_i|$ and,
    \begin{align*}
        d_H(U_1,\dots, U_r) < \frac{d}{2}.
    \end{align*}
    Let us partition $V_i$ using $U_i$. Let,
    \begin{align*}
        U_i(0) := U_i \quad \text{and} \quad  U_i(1) := V_i \setminus U_i.
    \end{align*}
    We will use $\mathbf{v}$ to denote vectors in $\{0,1\}^r$ and in particular let $\mathbf{0} = (0,\dots, 0).$ Further, we denote the components of a vector $\mathbf{v}$ by $(v_1,\dots, v_r)$. \\[0.3cm]
   Assume that for every $\mathbf{v}\in \{0,1\}^r\setminus\{\mathbf{0}\}$, 
    \begin{align*}
        d_H(U_1{(v_1)},\dots, U_r{(v_r)})< d\left(1+\frac{\eps^r}{2}\right).
    \end{align*}
    Then in view of the fact that $|U_i|\geq \eps|V_i|$ and $d_H(U_1,\dots, U_r)<d/2$, on counting the number of edges in $H[V_1,\dots,V_r]$ over all the parts $U_1{(v_1)},\dots, U_r{(v_r)}$, we have, 
    \begin{align*}
         e_H(V_1,\dots, V_r) &=\sum_{\mathbf{v}\in\{0,1\}^r}  d_H(U_1{(v_1)},\dots, U_r{(v_r)}) |U_1{(v_1)}|\cdots |U_r{(v_r)}|\\
         &< d\left(1+\frac{\eps^r}{2}\right)\left(\sum_{\mathbf{v}\in\{0,1\}^r\setminus\{\mathbf{0}\}}|U_1{(v_1)}|\cdots |U_r{(v_r)}|\right) + \frac{d}{2}|U_1{(0)}|\cdots |U_r{(0)}|\\
         &=d\left(1+\frac{\eps^r}{2}\right)\left(\sum_{\mathbf{v}\in\{0,1\}^r\setminus\{\mathbf{0}\}}|U_1{(v_1)}|\cdots |U_r{(v_r)}|\right) + \left(d-\frac{d}{2}\right)|U_1|\cdots |U_r|\\
         &< d|V_1|\cdots|V_r| + \eps^r\frac{d}{2}\left(1 - \frac{|U_1|\cdots |U_r|}{\eps^r|V_1|\cdots|V_r|}\right)|V_1|\cdots|V_r|\\ 
         &< d|V_1|\cdots|V_r|,
    \end{align*}
    which contradicts the assumption that $d_H(V_1,\dots,V_r)\geq d$. Thus, there exists some $\mathbf{w}\in \{0,1\}^r\setminus\{\mathbf{0}\}$, such that,
    \begin{align*}
        d_1:= d_H(U_1{(w_1)},\dots, U_r{(w_r)})\geq  d\left(1+\frac{\eps^r}{2}\right).
    \end{align*}
    Let these subsets corresponding to $\mathbf{w}$ be the first subset of vertices, i.e., 
    \begin{align*}
        U_1^{(1)}:= U_1(w_1), \dots, U_r^{(1)}:= U_r(w_r), \text{ and we have }  |U_i^{(1)}|\geq \eps|V_i| \text{ for every } i\in [r].  
    \end{align*}
    If $U_1^{(1)},\dots, U_r^{(1)}$ is not $(\eps,d/2)$-dense with respect to $H$, we repeat the above step with $U_1^{(1)},\dots, U_r^{(1)}$ playing the role of $V_1,\dots, V_r$. Continuing for $s$ steps, we obtain pairwise disjoint subsets, 
     \begin{align}
     \label{eqn: sizeofUi}
        U_1^{(s)}, \dots, U_r^{(s)} \text{ such that }  |U_i^{(s)}|\geq \eps|U_i^{(s-1)}| \text{ for every } i\in [r].  
    \end{align}
    Further, we have an increase in density, i.e., 
    \begin{align*}
        d_{s}:= d_H(U_1^{(s)}, \dots, U_r^{(s)}) \geq d\left(1+\frac{\eps^r}{2}\right)^{s}. 
    \end{align*}
    Since $d_s<1$, the process must stop before $s=s_0$ steps, where
    \begin{align*}
        s_0 \leq \frac{\ln\left(\frac{1}{d}\right)}{\ln\left(1+\frac{\eps^r}{2}\right)}.
    \end{align*}
    Using the inequality, 
    \begin{align*}
     x/2 \leq \ln(1+x) \text{ for } 0\leq x<1,
    \end{align*}
    we have,
     \begin{align}
    \label{eqn: s_0}
        s_0 \leq \frac{\ln\left(\frac{1}{d}\right)}{\ln\left(1+\frac{\eps^r}{2}\right)} \leq \frac{4\ln(\frac{1}{d})}{\eps^r}.
    \end{align}
    Thus the sets $U_1^{(s_0)},\dots, U_r^{(s_0)}$ are $(\eps,d/2)$-dense with respect to $H$ and further in view of (\ref{eqn: sizeofUi}), 
    \begin{align*}
        |U_i^{(s_0)}|\geq \eps^{s_0}|V_i| \text{ for every } i\in [r].  
    \end{align*}
    In view of the upper bound on $s_0$ in (\ref{eqn: s_0}), we have
    \begin{align*}
        |U_i^{(s_0)}|\geq \eps^{s_0}|V_i| \geq d^{\frac{4}{\eps^r}\ln(1/\eps)}|V_i| \text{ for every } i\in [r]. 
    \end{align*}
    \end{proof}

\section{Proof of Lemma \ref{lem: embedding}}

In this section we prove a counting version of Lemma \ref{lem: embedding}. Our proof is an adaptation of the argument in \cite{KNR10}. One of the main tools for our proof is the well-known graph counting lemma (see \cite{KSSS02}).

\begin{lemma}[Graph counting lemma]\label{lem:graphcounting}
Let $\ell\geq 1$ be an integer and $0<d, \epsilon<1$ real numbers satisfying $\epsilon^{1/\ell}<d<1-\epsilon^{1/\ell}$. Then the following holds for every graph $F$ on vertex set $[\ell]$. If $G$ is an $\ell$-partite graph with vertex partition $V(G)=V_1\cup\ldots \cup V_\ell$ such that $(V_i,V_j)$ is $(\epsilon,d)$-regular for every $1\leq i,j \leq \ell$, then the number of transversal induced copies of $F$ is
\begin{align*}
    (1\pm \epsilon)d^{e(F)}(1-d)^{\binom{\ell}{2}-e(F)}\prod_{i=1}^{\ell}|V_i|.
\end{align*}
\end{lemma}

We now state and proof the counting version of Lemma \ref{lem: embedding}.

\begin{lemma}
Let $n\geq 1$, $0<d<1$ and $0<\epsilon<d^{n^2}$. Given a graph $F$ and a linear $3$-graph $\cT$ with $\cT\subseteq \cK_3(F)$ on the same set of $n$ vertices, the following holds. For all $n$-partite graphs $G$ and $n$-partite 3-uniform hypergraphs $\cH$ with vertex set $V_1\cup\ldots\cup V_n$, $|V_i|> \epsilon^{-1}$, satisfying
\begin{enumerate}
    \item[(1)] $G$ is an underlying graph of $\cH$, i.e., $\cH\subseteq \cK_3(G)$.
    \item[(2)] For all $1\leq i < j\leq n$, the pairs $(V_i,V_j)$ are $(\eps,1/2)$-regular in $G$. 
    \item[(3)] For all $1\leq i < j < k\leq n$ are $(\eps, d)$-dense in $\cH$, 
\end{enumerate}
there exists at least
\begin{align*}
(1-\epsilon)\left(\frac{d}{2}\right)^{e(\cT)}\left(\frac{1}{2}\right)^{\binom{n}{2}}\prod_{i=1}^n |V_i|
\end{align*}
transversal copies of the system $(F,\cT)$ such that $F$ is induced in $G$ and $\cT\subseteq \cH$.
\end{lemma}

\begin{proof}
Write $V(F)=\{x_1,\ldots,x_n\}$, $m=e(\cT)$ and $\cT=\{T_1,\ldots,T_m\}$, where $T_i$'s are triangles of $F$. Our goal is to embed a transversal copy of $(F,\cT)$, i.e., a copy such that the vertex $x_i$ is sent to a vertex in $V_i$ for $1\leq i \leq n$. We will do that by induction on the number of edges of $\cT$. For $1\leq \ell \leq m$, let $N_{\ell}$ be the number of copies of $(F,\cT_\ell)$ with $F$ induced in $G$, $\cT_\ell:=\{T_1,\ldots,T_{\ell}\}\subseteq \cH$. We claim that
\begin{align}\label{eq:copies}
    N_{\ell}\geq  (1-\epsilon)\left(\frac{d}{2}\right)^{\ell}\left(\frac{1}{2}\right)^{\binom{n}{2}}\prod_{i=1}^n |V_i|
\end{align}
The base case $\ell=0$ follows immediately by Lemma \ref{lem:graphcounting}. Suppose that $N_{\ell}$ satisfies (\ref{eq:copies}). We now want to embed the edge $T_{\ell+1}$ in $\cH$. 
\\[0.3cm]
We may assume without lost of generality that $T_{\ell+1}=\{x_1,x_2,x_3\}$. Let $(\tilde{F},\tilde{\cT_\ell})$ be a copy of $(F,\cT)$ with vertex set $V:=V(\tilde{F})=\{y_1,\ldots, y_n\}$ in $(G,\cH)$. Let $F^*=\tilde{F}[V\setminus\{y_1,y_2,y_3\}]$. We define the set 
\begin{align*}
    \Ext(F^*)=\{(z_1,z_2,z_3) \in V_1\times V_2 \times V_3:\: G[V(F^*)\cup\{z_1,z_2,z_3\}] \text{ is a copy of $F$}\\ \text{and $\cH[V(F^*)\cup\{z_1,z_2,z_3\}]$ is a copy of $\cT_\ell$}\}  
\end{align*}
as the possible extensions of $F^*$ to a copy of the system $(F,\cT_\ell)$. In particular, by definition, $(y_1,y_2,y_3)\in \Ext(F^*)$. With this notation, $N_{\ell+1}$ is given by
\begin{align*}
\sum_{F^*}\sum_{(z_1,z_2,z_3)\in \Ext(F^*)}\ind_{\cH}(z_1,z_2,z_3),
\end{align*}
where $\ind_{\cH}(z_1,z_2,z_3)$ is one if $\{z_1,z_2,z_3\}\in \cH$ and zero otherwise.
\\[0.3cm]
Let $W_1^{F^*}, W_2^{F^*}$ and $W_3^{F^*}$ be the set of possible values of $z_1, z_2$ and $z_3$ in $\Ext(F^*)$, respectively. Since $\cT_{\ell}$ is a linear $3$-graph, there are no edge dependencies between $z_i$'s and consequently $\Ext(F^*)=W_1\times W_2\times W_3$. Hence, 
\begin{align*}
    N_{\ell+1}=\sum_{F^*}e_{\cH}(W_1^{F^*},W_2^{F^*},W_3^{F^*}).
\end{align*}
We split the sum into two parts. Let 
\begin{align*}
    \cS&=\{F^*:\: |W_1^{F^*}||W_2^{F^*}||W_3^{F^*}|<\epsilon|V_1||V_2||V_3|\}\\
     \cL&=\{F^*:\: |W_1^{F^*}||W_2^{F^*}||W_3^{F^*}|\geq\epsilon|V_1||V_2||V_3|\}
\end{align*}
be the set of $F^*$ that $|W_1^{F^*}||W_2^{F^*}||W_3^{F^*}|$ is small and large, respectively. Note that since $(V_1,V_2,V_3)$ is $(\epsilon,d)$ dense, it follows that for $F^*\in \cL$
\begin{align}\label{eq:big1}
    e_{\cH}(W_1^{F^*},W_2^{F^*},W_3^{F^*})\geq d|W_1^{F^*}||W_2^{F^*}||W_3^{F^*}|.
\end{align}
Moreover, a computation shows that
\begin{align*}
    (1-\epsilon)\left(\frac{d}{2}\right)^{\ell}\left(\frac{1}{2}\right)^{\binom{n}{2}}\prod_{i=1}^n |V_i|&\leq N_{\ell}=\sum_{F^*}|W_1^{F^*}||W_2^{F^*}||W_3^{F^*}|\\
    &\leq\epsilon\prod_{i=1}^n |V_i|+\sum_{F^*\in \cL}|W_1^{F^*}||W_2^{F^*}||W_3^{F^*}|
\end{align*}
and since $\epsilon<d^{n^2}$ we obtain that 
\begin{align}\label{eq:big2}
    \sum_{F^*\in \cL}|W_1^{F^*}||W_2^{F^*}||W_3^{F^*}|\geq (1-\epsilon)\frac{1}{2}\left(\frac{d}{2}\right)^{\ell}\left(\frac{1}{2}\right)^{\binom{n}{2}}\prod_{i=1}^n |V_i|
\end{align}
Relations (\ref{eq:big1}) and (\ref{eq:big2}) gives us that
\begin{align*}
N_{\ell+1}&=\sum_{F^*}e_{\cH}(W_1^{F^*},W_2^{F^*},W_3^{F^*})\geq \sum_{F^*\in\cL}e_{\cH}(W_1^{F^*},W_2^{F^*},W_3^{F^*})\\
&\geq d \sum_{F^*\in \cL}|W_1^{F^*}||W_2^{F^*}||W_3^{F^*}|\geq (1-\epsilon)\left(\frac{d}{2}\right)^{\ell+1}\left(\frac{1}{2}\right)^{\binom{n}{2}}\prod_{i=1}^n |V_i|,
\end{align*}
which concludes the induction and the proof of the lemma.
\end{proof}
\end{document}